\documentclass{article}
\usepackage[utf8]{inputenc}

\usepackage{hyperref,mathrsfs,enumitem,amsmath, amssymb, amsthm, graphicx,color}

\newcommand{\vertiii}[1]{{\left\vert\kern-0.25ex\left\vert\kern-0.25ex\le
ft\vert #1 
    \right\vert\kern-0.25ex\right\vert\kern-0.25ex\right\vert}}

\newcommand{\N}{\mathbb{N}}

\newcommand{\C}{\mathbb{C}}

\newcommand{\overn}{\overset{n\to\infty}{\longrightarrow}}

\newcommand{\ds}{\displaystyle}
\usepackage{mathtools}

\newtheorem{theorem}{Theorem}
\newtheorem{lemma}[theorem]{Lemma}

\begin{document}
\title{Two Families of Hypercyclic Non-Convolution Operators}
\author{Alex Myers\footnote{St. Olaf College, Northfield, MN 55057}, Muhammadyusuf Odinaev\textsuperscript{$*$},David Walmsley\footnote{Department of Mathematics, Statistics and Computer Science, St. Olaf College, Northfield, MN 55057}}




\date{}
\maketitle

\begin{abstract}
Let $H(\C)$ be the set of all entire functions endowed with the topology of uniform convergence on compact sets. Let $\lambda,b\in\C$, let $C_{\lambda,b}:H(\C)\to H(\C)$ be the composition operator $C_{\lambda,b} f(z)=f(\lambda z+b)$, and let $D$ be the derivative operator. We extend results on the hypercyclicity of the non-convolution operators $T_{\lambda,b}=C_{\lambda,b} \circ D$ by showing that whenever $|\lambda|\geq 1$, the collection of operators
\begin{align*}
    \{\psi(T_{\lambda,b}): \psi(z)\in H(\C), \psi(0)=0 \text{ and } \psi(T_{\lambda,b}) \text{ is continuous}\}
\end{align*}
forms an algebra under the usual addition and multiplication of operators which consists entirely of hypercyclic operators (i.e., each operator has a dense orbit). We also show that the collection of operators 
\begin{align*}
    \{C_{\lambda,b}\circ\varphi(D): \varphi(z) \text{ is an entire function of exponential type with } \varphi(0)=0\}
\end{align*}
consists entirely of hypercyclic operators.
\end{abstract}

\section{Introduction}\label{Introduction}
Let $\C$ denote the complex plane and $H(\C)$ be the set of all entire functions endowed with the topology of uniform convergence on compact sets. This topology makes $H(\C)$ a separable Fr\'echet space, which is a locally convex and metrizable topological vector space that is both complete and separable. A continuous linear operator $T$ defined on a Fr\'echet Space $\mathcal{F}$ is said to be hypercyclic if there exists $f\in \mathcal{F}$ (called a hypercyclic vector for $T$) such that the orbit $\{T^n f: n\in \N\}$ is dense in $\mathcal{F}$. We refer the reader to the books \cite{GrosseErdmann} and \cite{BayartMatheron} for a thorough introduction to the study of hypercyclic operators.

The first example of a hypercyclic operator was given by Birkhoff in 1929, who showed that the translation operator $T:f(z)\mapsto f(z+1)$ is hypercyclic \cite{Birkhoff}. In 1952, MacLane showed that the derivative operator $D:f(z)\mapsto f'(z)$ is also hypercyclic \cite{MacLane}. Both of these results were unified in a substantial paper by Godefroy and Shapiro in 1991, who proved that every continuous linear operator $L:H(\C)\to H(\C)$ which commutes with translations (these operators are called convolution operators) and which is not a scalar multiple of the identity is hypercyclic \cite{GodefroyShapiro}.

It is then natural to ask for examples of hypercyclic non-convolution operators, that is, operators which are hypercyclic but do not commute with all translations. This study was initiated by Aron and Markose in 2004 in \cite{AronMarkose}, where they introduced the non-convolution operators $T_{\lambda,b} f(z) = f'(\lambda z + b)$, $\lambda,b\in \C$ and proved, along with the authors in \cite{Hallack}, that such operators are hypercyclic when $|\lambda|\geq 1$. Their result was extended by Le\'on-Saavedra and Romero de la Rosa to show, among other things, that such operators are not hypercyclic when $|\lambda|<1$ in \cite{LeonRosa}.  An $N$-dimensional analogue of these operators was studied in \cite{Muro}, and other examples of hypercyclic non-convolution operators can be found in \cite{Vitaly,Petersson}.  The purpose of this paper is to extend and complement the examples of hypercyclic non-convolution operators given in \cite{AronMarkose,Hallack,LeonRosa}.

This paper is organized as follows. In Section \ref{Poly}, we establish some lemmas and prove that if $\psi(z)$ is an entire function such that $\psi(0)=0$ and $\psi(T_{\lambda,b})$ is continuous, then the operator $\psi(T_{\lambda,b})$ is hypercyclic if and only if $|\lambda|\geq 1$. Our other main result is in Section \ref{last section}, the description of which requires some established notation. Let $D$ be the derivative operator. Godefroy and Shapiro showed that an operator $V$ on $H(\C)$ is a convolution operator if and only if $V=\varphi(D)$, where $\varphi(z)$ is an entire function of exponential type. For $\lambda,b\in \C$, define the composition operator $C_{\lambda,b}:H(\C)\to H(\C)$ by $C_{\lambda,b} f(z)=f(\lambda z+b)$. If $\varphi(D)$ is a convolution operator, define by $L_{\lambda,b,\varphi}$ the operator $L_{\lambda,b,\varphi} = C_{\lambda,b} \circ \varphi(D)$. If $\varphi(z)=z$, then $L_{\lambda,b,\varphi}=T_{\lambda,b}$, so that these operators $L_{\lambda,b,\varphi}$ generalize those introduced by Aron and Markose. In Section 3, we prove that if $\varphi(0)=0$ and $|\lambda|\geq 1$, then $L_{\lambda,b,\varphi}$ is hypercyclic.

\section{The hypercyclicity of polynomials of $T_{\lambda,b}$}\label{Poly}

To show that an operator $T$ on $H(\C)$ is hypercyclic, we will employ the well-known Hypercyclicity Criterion. It states, for our purposes, that a continuous linear operator $T$ is hypercyclic on $H(\C)$ if there exists a dense set $\mathcal{P}\subset H(\C)$ and a sequence of mappings $S_n:\mathcal{P}\to H(\C)$ such that
\begin{enumerate}[label={\upshape(\alph*)}]
    \item $T^n f \to 0$ for all $f\in \mathcal{P}$,
    \item $S_n f \to 0$ for all $f\in \mathcal{P}$, and
    \item $T^n S_n f \to f$  for all $f\in \mathcal{P}$.
\end{enumerate}
Actually, there are more general conditions which ensure the hypercyclicity of an operator, but we will not use them.  We refer the reader to \cite[Chapter 3]{GrosseErdmann} for more details about the Hypercyclicity Criterion.

As seen in condition (c) above, the mappings $S_n$ act almost as right inverses (as $n\to \infty$) for $T^n$ on $\mathcal{P}$. Borrowing the ideas in \cite[Lemma 1]{Chan}, we begin with a crucial lemma that will help determine these mappings for a large class of operators. Let $I$ be the identity operator on $H(\C)$ and $\mathcal{P}$ be the collection of complex polynomials in $H(\C)$, which is a dense subset of $H(\C)$.

\begin{lemma}\label{InverseLemma}
Let $\psi(z)=\sum_{k=0}^\infty w_k z^k$ be an entire function such that $\psi(0)\not =0$.  Suppose $G:H(\C)\to H(\C)$ is an operator such that
\begin{enumerate}[label={\upshape(\alph*)}]
    \item $G(\mathcal{P})\subseteq \mathcal{P}$,
    \item $\deg Gp < \deg p$ for all nonzero $p\in\mathcal{P}$, and
    \item for all $\lambda\in \C$, $\psi(\lambda G) =\sum_{k=0}^\infty w_k \lambda^k G^k$ is a continuous linear operator on $H(\C)$.
\end{enumerate}
Then for each nonzero $\lambda\in\C$, there exists a right-inverse mapping $S_{\psi(\lambda G)}:\mathcal{P}\to\mathcal{P}$ for $\psi(\lambda G)$ such that $\psi(\lambda G)S_{\psi(\lambda G)} p = p $ for all $p\in \mathcal{P}$. Moreover, for each non-negative integer $m$, there exists $C=C(m)>0$ such that for each positive integer $n$ there exist constants $a_{i,n}$, $1\leq i\leq m$, such that for each nonzero $\lambda\in \C$ and each nonzero polynomial $p\in \mathcal{P}$ of degree $m$, $S_{\psi(\lambda G)}^n p$ has the form
\[ S_{\psi(\lambda G)}^n p = w_0^{-n}\left(I + a_{1,n}G + \cdots + a_{m,n} G^m\right) p, \]
and the coefficients $a_{i,n}$ satisfy $|a_{i,n}|<Cn^m$ for all $i$ and $n$.
\end{lemma}

\begin{proof}
Let $p$ be a nonzero polynomial, let $m$ be the degree of $p$, and let $\lambda\in\C$ be nonzero. Since $\deg Gp < \deg p$, we have that $\psi(\lambda G) p = \sum_{k=0}^m w_k \lambda^k G^k p$.

Let $\alpha_1,\alpha_2,\ldots,\alpha_m$ be the zeros of the polynomial $q(z)=w_0+w_1 z + \cdots + w_m z^m$, repeated according to multiplicity. Since $w_0\not = 0$, none of the $ \alpha_i $ equal zero, and thus $q(z)=a_0(1-z/\alpha_1)\cdots (1-z/\alpha_m)$. Hence we can write 
\begin{align}\notag
\psi(\lambda G) p = w_0\left(I-\frac{\lambda G}{\alpha_1}\right)\left(I-\frac{\lambda G}{\alpha_2}\right)\cdots\left(I-\frac{\lambda G}{\alpha_m}\right)p.
\end{align}
We find a right-inverse mapping for each factor $\left(I-\frac{\lambda G}{\alpha_i}\right)$ as follows. Let $i$ be an integer satisfying $0\leq i\leq m$. Since $\deg Gp < \deg p$,
\begin{align*}
p &= \left(I-\left(\frac{\lambda G}{\alpha_i}\right)^{m+1}\right)p = \left(I-\frac{\lambda G}{\alpha_i}\right)\left(I+\frac{\lambda G}{\alpha_i}+ \frac{\lambda^{2} G^{2}}{\alpha_i^{2}} + \cdots + \frac{\lambda^{m} G^{m}}{\alpha_i^{m}}\right)p.
\end{align*}
Thus define $S_i p = \left(I+\frac{\lambda G}{\alpha_i}+ \frac{\lambda^{2} G^{2}}{\alpha_i^{2}} + \cdots + \frac{\lambda^{m} G^{m}}{\alpha_i^{m}}\right)p$, and observe $S_i p$ is a polynomial of degree $m$. We then define $S_{\psi(\lambda G)} p$ as $S_{\psi(\lambda G)}p=\frac{1}{w_0}S_1\cdots S_m p$.  Thus $S_{\psi(\lambda G)}p$ has degree $m$ and $\psi(\lambda G)S_{\psi(\lambda G)}p=p$.

By writing the formula for $S_{\psi(\lambda G)}p$ as
\begin{align}\notag
S_{\psi(\lambda G)}p = \frac{1}{w_0}\left( \displaystyle\prod\limits_{i=1}^{m} \left(I+\frac{\lambda G}{\alpha_i}+ \frac{\lambda^{2} G^{2}}{\alpha_i^{2}} + \cdots + \frac{\lambda^{m} G^{m}}{\alpha_i^{m}}\right) \right) p,
\end{align}
we obtain the form for $S_{\psi(\lambda G)}^n p$ by multiplying out the above product and keeping only the terms involving $G^i$ for $0\leq i \leq m$ to obtain
\begin{align*}
    S_{\psi(\lambda G)}^n p = a_0^{-n}\left(I + a_{1,n} \lambda G + \cdots + a_{m,n} \lambda^m G^m\right) p.
\end{align*}
Let $r=\max\{1,|\alpha_1|^{-1},|\alpha_2|^{-1},\ldots,|\alpha_m|^{-1}\}$ and let $C(mn,i)$ be the coefficient of $y^i$ in the expansion of $(1+y+y^2+y^3+...)^{mn}$. Then $|a_{i,n}| \leq r^i C(mn,i)$.  Since $(1+y+y^2+\dots)^{mn} = 1/(1-y)^{mn}$ for $y \in (-1,1)$, and the Taylor Series for $1/(1-y)^{mn}$ is 
\[1+mny+\frac{mn(mn+1)y^2}{2!}+\frac{mn(mn+1)(mn+2)y^2}{3!}+\dots,\]
we have that 
\[C(mn,i) =\binom{mn+i-1}{i} =  \left( \frac{mn}{1}\right) \left( \frac{mn+1}{2}\right)\cdots \left( \frac{mn+i-1}{i}\right) \leq (mn)^i,\]
which implies that 
\begin{align}\label{inverse coeff}
    |a_{i,n}|\leq r^i C(mn,i) \leq (rmn)^i \leq (rmn)^m.
\end{align}
By taking $C=(rm)^m$, the proof is complete.
\end{proof} 

To simplify our presentation, let us recall some notation from above and establish a bit more for the remainder of this paper. Let $\lambda,b\in\C$, and recall that we define the operator $C_{\lambda,b}:H(\C)\to H(\C)$ by $C_{\lambda,b} f(z) = f(\lambda z+b)$, and the operator $T_{\lambda,b}: H(\C)\to H(\C)$ by $T_{\lambda, b}f(z)=f'(\lambda z+b)$. With this notation, we can view $T_{\lambda,b}$ as a composition of two operators, namely
\begin{align*}
    T_{\lambda,b}=C_{\lambda,b} \circ D.
\end{align*}
More generally, if $\varphi(z)$ is an entire function of exponential type, then we compose $C_{\lambda,b}$ and the convolution operator $\varphi(D)$ to define the operator $L_{\lambda,b,\varphi}$ as
\begin{align*}
    L_{\lambda,b,\varphi}=C_{\lambda,b} \circ \varphi(D).
\end{align*}
Observe that if $\varphi(z)=z$, then $L_{\lambda,b,\varphi}=T_{\lambda,b}$. Our immediate focus will turn to the operators $L_{\lambda,b,z^m}=C_{\lambda,b}\circ D^m$, where $m\in \N$. 

In this case, one can check that $C_{\lambda,b}^n f(z) = f\left(\lambda^n z + \frac{1-\lambda^n}{1-\lambda }b\right) \coloneqq f(\lambda^n z - r_n)$, where $r_n=-b\frac{1-\lambda^n}{1-\lambda }$. A straightforward induction argument then yields that
\begin{align}\label{L^n eqn}
    L_{\lambda,b,z^m}^n f(z) = \lambda^{\frac{mn(n-1)}{2}} C_{\lambda,b}^n \circ D^{mn} f(z) = \lambda^{\frac{mn(n-1)}{2}} f^{(mn)}(\lambda^n z -r_n).
\end{align}
We want to use the Hypercyclicity Criterion to show that these operators are hypercyclic when $|\lambda|\geq 1$. The bulk of the work, as usual, is to determine the right-inverse mappings. To motivate what follows, let us describe a derivation of these mappings by closely following the work in \cite{Hallack} to determine a sequence of right-inverse mappings $S_{m,n}$ for $L_{\lambda,b,z^m}^n$ defined on the set of complex polynomials $\mathcal{P}$. 

Define a ``formal operator" $A_m$ on the set $\mathcal{D}=\{1\}\cup \{d(z+c)^k: k\in \N, c,d\in \C\}$ by
\begin{align*}
    A_m(1) = \frac{(z-b)^m}{m!}, \text{ and} \\
    A_m(d(z+c)^k) = \frac{k!d(z+c)^{k+m}}{(k+m)!}.
\end{align*}
The formal operator $A_m$ acts as an ``$m$th antiderivative" operator, but it is just a formal tool we use to motivate well-defined right-inverse mappings $S_{m,n}$. With a formal antiderivative operator at hand, it is then natural to try and define $S_{m,1}$ on the basis $\{1,z,z^2,\ldots\}$ for $\mathcal{P}$ as
\begin{align*}
    S_{m,1}(z^k)=A_m \circ C_{\lambda,b}^{-1} (z^k),
\end{align*}
and then take $S_{m,n}$ to be $(S_{m,1})^n$. This would yield
\begin{align}\label{first Sn eqn}
    S_{m,n}(z^k)=\frac{k!}{(k+mn)! \lambda^{kn}\lambda^{\frac{mn(n-1)}{2}}} (z+r_n)^{k+mn},
\end{align}
which works well with condition (c) of the Hypercyclicity Criterion since $L_{\lambda,b,z^m}^n \circ S_{m,n} = I$ on $\{1,z,z^2,\ldots\}$ for this choice of $S_{m,n}$. However, one would run into issues checking condition (b), as one would find high powers of $\lambda$ paired with low powers of $z$ when expanding $(z+r_n)^{k+mn}$. To deal with this issue, we will add a polynomial of degree less than $mn$ to the definition of $S_{m,n}$ in (\ref{first Sn eqn}) to ``kill off" these high powers of $\lambda$. This new definition of $S_{m,n}$ will then satisfy condition (b) of the Hypercyclicity Criterion, and still satisfy condition (c) because any polynomial of degree less than $mn$ belongs to the kernel of $L_{\lambda,b,z^m}^n$. We provide the details in the following lemma.

\begin{lemma}\label{HallackLemma}
Let $m\in \N$. There are linear mappings $S_{m,n}:\mathcal{P}\to\mathcal{P}$ defined by
\begin{align}\label{Sformula}
    S_{m,n}(z^k) = \frac{k!}{(k+mn)! \lambda^{kn}\lambda^{\frac{mn(n-1)}{2}}} \sum_{j=0}^k \binom{k+mn}{j} z^{k+mn-j} r_n^j,
\end{align}
such that for all $p\in \mathcal{P}$, $L_{\lambda,b,z^m}^n S_{m,n} p=p$ .

Furthermore, if $\{\sigma_n\}$ is a sequence of complex numbers for which there exists $t>0$ such that $|\sigma_n|\leq t^n$ for all $n\in\N$, then for all $p\in \mathcal{P}$ and for all $\ell\in \N$,
\begin{align}\label{Slimit}
    S_{m,\ell n}(\sigma_n p)\to 0.
\end{align}
\end{lemma}

\begin{proof}
For each monomial $z^k$, define $S_{m,n}(z^k)$ as above in (\ref{Sformula}), and extend $S_{m,n}$ linearly to $\mathcal{P}$.
Let $\Delta_{k,n}=\ds\frac{k!}{(k+mn)! \lambda^{kn}\lambda^{\frac{mn(n-1)}{2}}}(z+r_n)^{k+mn}-S_{m,n}(z^k)$. By expanding $(z+r_n)^{k+mn}$ using the Binomial Theorem and cancelling common terms, one can establish that \[\Delta_{k,n}=\ds \frac{k!}{(k+mn)! \lambda^{kn}\lambda^{\frac{mn(n-1)}{2}}} \sum_{j=k+1}^{k+mn} \binom{k+mn}{j} z^{k+mn-j} r_n^j.\] 
Thus the degree of $\Delta_{k,n}$ is $mn-1$, and hence it belongs to the kernel of $L_{\lambda,b,z^m}^n $ by equation (\ref{L^n eqn}).

Using equations (\ref{Sformula}) and (\ref{L^n eqn}), we compute that
\begin{align*}
    & L_{\lambda,b,z^m}^n S_{m,n} (z^k)\\
    & = L_{\lambda,b,z^m}^n \left( S_{m,n}(z^k)+\Delta_{k,n}\right)\\
    & = L_{\lambda,b,z^m}^n \left( \ds\frac{k!}{(k+mn)! \lambda^{kn}\lambda^{\frac{mn(n-1)}{2}}}(z+r_n)^{k+mn}\right)\\
    & =  \frac{k!}{(k+mn)! \lambda^{kn}\lambda^{\frac{mn(n-1)}{2}}} \cdot \frac{\lambda^{\frac{mn(n-1)}{2}}(k+mn)!}{k!} (\lambda^n z - r_n + r_n)^k \text{ by (\ref{L^n eqn})}\\
    & = z^k.
\end{align*}
This verifies that $S_{m,n}$ is a right inverse for $L_{\lambda,b,z^m}^n$ on $\mathcal{P}$.

To verify the limit in (\ref{Slimit}), let $p$ be a polynomial and let $\ell \in \N$. By the linearity of $S_{m,\ell n}$ it suffices to show that 
\begin{align}\label{reduce1}
    S_{m,\ell n}(\sigma_n z^k) \to 0
\end{align}
uniformly on compact subsets of $\C$ for each monomial $z^k$.  Let $R>0$, let $|z|\leq R$, and let $j$ be an integer satisfying $0\leq j\leq k$. As $S_{m,\ell n}(z^k)$ is a sum of $k+1$ terms, to show (\ref{reduce1}) it suffices to show  that each of those terms converges to zero uniformly when $|z|\leq R$. That is, it suffices to show
\begin{align}\label{reduce2}
    \left|\frac{\sigma_n k!}{(k+m\ell n)! \lambda^{k\ell n}\lambda^{\frac{m\ell n( \ell n-1)}{2}}}  \binom{k+m\ell n}{j} z^{k+m\ell n-j} r_{\ell n}^j\right|\to 0
\end{align}
uniformly.  Since $|\lambda|\geq 1$,
\begin{align}\label{rnbound}
    |r_{\ell n}|=|b|\left|\sum_{i=0}^{\ell n-1} \lambda^i\right|\leq \ell n|b|\cdot |\lambda|^{\ell n-1}. 
\end{align}
Thus
\begin{align*}
     & \left|\frac{\sigma_n k!}{(k+m\ell n)! \lambda^{k\ell n}\lambda^{\frac{m\ell n( \ell n-1)}{2}}}  \binom{k+m\ell n}{j} z^{k+m\ell n-j} r_{\ell n}^j\right|\\
     & \leq \frac{t^n k! R^{k+m\ell n-j} (\ell n)^j |b|^j |\lambda|^{(\ell n-1)j}}{(k+m\ell n)! |\lambda|^{k\ell n}|\lambda|^{\frac{m\ell n(\ell n-1)}{2}}}  \binom{k+m\ell n}{j} \text{ by (\ref{rnbound})}\\
     & = \frac{t^n k! R^{k+m\ell n-j} (\ell n)^j |b|^j |\lambda|^{(\ell n-1)j}}{(k+m\ell n)! |\lambda|^{k\ell n}|\lambda|^{\frac{m\ell n(\ell n-1)}{2}}}  \frac{(k+m\ell n)!}{j!(k+m\ell n-j)!}  \overn 0,
\end{align*}
which establishes the limit in (\ref{reduce2}) and completes the proof.
\end{proof}

With the previous lemmas at hand, we can now use the Hypercyclicity Criterion to show that there are many hypercyclic non-convolution operators that can be generated by $T_{\lambda,b}$.
\begin{theorem}\label{p(T)}
Let $\lambda,b\in\C$ with $|\lambda|\geq 1$, and let $T_{\lambda,b}:H(\C)\to H(\C)$ be the operator defined by $T_{\lambda,b}:f(z)\mapsto f'(\lambda z +b)$. If $\psi$ is an entire function such that $\psi(0)=0$ and $\psi(T_{\lambda,b})$ is a continuous linear operator, then $\psi(T_{\lambda,b})$ is hypercyclic.
\end{theorem}

\begin{proof}
Let $\psi(z)=\xi(z)z^\ell$, where $\ell \in\N$ and $\xi(z)$ is an entire function with $\xi(0)\not =0$. Then the operator $\psi(T_{\lambda,b})=\xi(T_{\lambda,b})T_{\lambda,b}^\ell$, and $(\psi(T_{\lambda,b}))^n=(\xi(T_{\lambda,b}))^n T_{\lambda,b}^{\ell n}$.  Let $\mathcal{P}$ be the set of complex polynomials in $H(\C)$, which is a dense subset of $H(\C)$, and let $g\in \mathcal{P}$. Since $T_{\lambda,b}(g)\in \mathcal{P}$ and $\deg T_{\lambda,b}(g) < \deg g$, by Lemma \ref{InverseLemma} there is a mapping $S_{\xi(T_{\lambda,b})}:\mathcal{P}\to\mathcal{P}$ such that $\xi(T_{\lambda,b})^n S_{\xi(T_{\lambda,b})}^n (g)=g$ for all $g\in \mathcal{P}$. Since $T_{\lambda,b}=L_{\lambda,b,z}$, by Lemma \ref{HallackLemma} there exist linear mappings $S_{1,\ell n}:\mathcal{P}\to\mathcal{P}$ such that $T_{\lambda,b}^{\ell n} S_{1,\ell n}g=g$ for all $g\in\mathcal{P}$. Thus for all $g\in \mathcal{P}$,
\begin{align*}
    \psi(T_{\lambda,b})^n S_{1,\ell n}S_{\xi(T_{\lambda,b})}^n (g) = \xi(T_{\lambda,b})^nT_{\lambda,b}^{\ell n}S_{1,\ell n}S_{\xi(T_{\lambda,b})}^n (g) = g,
\end{align*}
so the mapping $S_{1,\ell n}S_{\xi(T_{\lambda,b})}^n$ is a right inverse for $\psi(T_{\lambda,b})^n$ on $\mathcal{P}$.

We check the three conditions of the Hypercyclicity Criterion. As just mentioned, the third condition is satisfied.  Let $p$ be a polynomial of degree $d$.  Since $T_{\lambda,b}^n p=0$ whenever $n>d$, and since $\psi(0)=0$, we have that $(\psi(T_{\lambda,b}))^n p=0$ whenever $n>d$. Thus the first condition of the Hypercyclicity Criterion is satisfied.

What remains to check is the second condition.   Let $\xi(0)=w_0$. By Lemma \ref{InverseLemma}, there is a constant $C=C(d)>0$ and are constants $a_{i,n}$ for $1\leq i\leq d$ such that
\begin{align}\label{Sq}
    S_{\xi(T_{\lambda,b})}^n p  = w_0^{-n}\left(I + a_{1,n}T_{\lambda,b} + \cdots + a_{d,n} T_{\lambda,b}^d\right) p,
\end{align}
and $|a_{i,n}|<Cn^d$ for all $i$ and $n$. For each positive integer $n$, let $a_{0,n}=1$. Let $i$ be an integer such that $0\leq i\leq d$. 

By the linearity of $S_{1,\ell n}$ and equation (\ref{Sq}), to show 
$S_{1,\ell n}S_{\xi(T_{\lambda,b})}^n p\overn 0$ uniformly on compact subsets of $\C$, it suffices to show that 
\begin{align}\label{endgoal}
    S_{1,\ell n} (w_0^{-n}a_{i,n} T_{\lambda,b}^i p)\overn 0
\end{align}
uniformly on compact subsets of $\C$. If $i=0$, let $\sigma_n=w_0^{-n}$ and $t=|w_0|^{-1}$. Then the limit $(\ref{Slimit})$ in Lemma \ref{HallackLemma} implies (\ref{endgoal}). If $i>0$, let $\sigma_n=w_0^{-n}a_{i,n}$. Since $|w_0^{-n}a_{i,n}|\leq |w_0|^{-n} Cn^d<t^n$ for some $t>0$, limit (\ref{Slimit}) in Lemma \ref{HallackLemma} again yields (\ref{endgoal}), as desired. This shows that the second condition of the Hypercyclicity Criterion is satisfied. Hence $\psi(T_{\lambda,b})$ is hypercyclic.
\end{proof}

Theorem \ref{p(T)} really provides an algebra of hypercyclic non-convolution operators when $|\lambda|\geq 1$ (except when $\lambda=1$, in which case the operators are indeed convolution operators). To see this, let $\psi_1(z),\psi_2(z)\in H(\C)$ such that $\psi_1(0)=0=\psi_2(0)$ and $\psi_1(T_{\lambda,b})$, $\psi_2(T_{\lambda,b})$ are continuous. Then $\psi_1(0)+\psi(0)=0, \psi_1(0)\psi_2(0)=0$, and $\psi_1(T_{\lambda,b})\psi_2(T_{\lambda,b})$ is a continuous operator. Hence $\psi_1(T_{\lambda,b})+\psi_2(T_{\lambda,b})$ and $\psi_1(T_{\lambda,b})\psi_2(T_{\lambda,b})$ are hypercyclic by Theorem \ref{p(T)}, and any non-zero scalar multiple of them is as well. We next show that $|\lambda|\geq 1$ is necessary for the hypercyclicity of these types of operators.
\begin{theorem}\label{nonhypercyclic}
Let $\lambda,b\in\C$ with $|\lambda|<1$, and let $T_{\lambda,b}:H(\C)\to H(\C)$ be the  operator defined by $T_{\lambda,b}:f(z)\mapsto f'(\lambda z +b)$. If $\psi(z)$ is an entire function such that $\psi(0)=0$ and $\psi(T_{\lambda,b})$ is a continuous linear operator, then $\psi(T_{\lambda,b})$ is not hypercyclic.
\end{theorem}

\begin{proof}
We will show that $(\psi(T_{\lambda,b}))^n f\overn 0$ for every $f\in H(\C)$. Since $\psi(0)=0$, we may write $\psi(z)=\alpha z^\ell \xi(z)$, where $\alpha\in \C\setminus\{0\}$ and $\xi(z)=\sum_{k=0}^\infty w_k z^k$ is an entire function such that $\xi(0)=w_0=1$. Let $c_{j,n}$ be the $j$th Taylor coefficient of the Taylor series of $(\xi(z))^n$ centered at zero, so that
\begin{align*}
    (\xi(z))^n = (1+w_1 z + w_2 z^2 + \cdots )^n \coloneqq 1+c_{1,n} z + c_{2,n} z^2  +\cdots,
\end{align*}
and let $r=\sup\{|w_k| :k\in\N\cup\{0\}\}$. Then by the same type of argument used in Lemma \ref{inverse coeff} to obtain the estimate (\ref{inverse coeff}), we have that 
\begin{align}\label{cjnBound}
    |c_{j,n}| \leq r^j\binom{n+j-1}{j}=r^j \left( \frac{n}{1}\right) \left( \frac{n+1}{2}\right)\cdots \left( \frac{n+j-1}{j}\right)\leq r^j n^j.
\end{align}

Now let $f\in H(\C)$, let $C=\max\{|f(z)|: |z|\leq 1\}$, let $R>0$ be given, and let $|z|\leq R$. As shown in the proof of \cite[Theorem 2.1]{LeonRosa}, there exists an $n_0\in \N$ such that for all $n\geq n_0$,
\begin{align}\label{LeonRosa}
    \left| T_{\lambda,b}^n f(z)\right| \leq C n! 2^n |\lambda|^{\frac{n(n-1)}{2}}.
\end{align}
Furthermore, since $\lim_{n\to\infty} 2rn^2\ell |\lambda|^{\ell n}=0$, there exists $n_1\in \N$ such that $n\geq n_1$ implies
\begin{align}\label{ratio}
    2rn^2\ell |\lambda|^{\ell n} <1.
\end{align}

Now let $n\geq \max\{n_0,n_1\}$, and let $c_{0,n}=1$. Then
\begin{align}\label{psi sum} \notag
  & \left| \left(\psi(T_{\lambda,b})\right)^n f(z)\right|\\ \notag
  & = \left|\alpha^n T^{n\ell} \left( I + c_{1,n}T  + c_{2,n}T^{2} + \cdots\right) f(z)\right| \\ \notag
   & \leq |\alpha|^n \sum_{j=0}^\infty \left| c_{j,n}T^{n\ell+j} f(z) \right|\\ 
   & \leq |\alpha|^n \sum_{j=0}^\infty  r^j n^j C(n\ell+j)!2^{n\ell+j}|\lambda|^{\frac{(n\ell+j)(n\ell+j-1)}{2}} \text{ by (\ref{cjnBound}) and (\ref{LeonRosa})}.
\end{align}
Let $\beta_{j,n}$ be the $j$th term in the sum in the previous line. Then (\ref{ratio}) implies
\begin{align}\label{ratio test}
    \left|\frac{\beta_{j+1,n}}{\beta_{j,n}}\right|=rn(n\ell +j+1)2|\lambda|^{j+\ell n}\leq rn(n\ell(j+2))2|\lambda|^{j+\ell n}<(j+2)|\lambda|^j.
\end{align}
We claim that inequality (\ref{ratio test}) implies that $|\beta_{j,n}|< (j+2)!|\lambda|^\frac{j(j-1)}{2} |\beta_{0,n}|$ for each $j\in\N\cup\{0\}$. To prove this, we proceed by induction. The basis for induction follows immediately upon substituting $j=0$ into (\ref{ratio test}). For the induction step, we assume that $|\beta_{j,n}|< (j+2)!|\lambda|^\frac{j(j-1)}{2} |\beta_{0,n}|$ is true for some $j\in \N\cup\{0\}$. Then
\begin{align*}
    |\beta_{j+1,n}| & < (j+3)|\lambda|^j |\beta_{j,n}| \text{ by (\ref{ratio})}\\
    & < (j+3)|\lambda|^j  (j+2)!|\lambda|^\frac{j(j-1)}{2}|\beta_{0,n}| \text{ by the induction hypothesis}\\
    & = (j+3)!|\lambda|^\frac{j(j+1)}{2}|\beta_{0,n}|,
\end{align*}
which had to be shown. This completes the proof of our claim.

Since $|\beta_{j,n}|< (j+2)!|\lambda|^\frac{j(j-1)}{2} |\beta_{0,n}|$ for each $j\in\N\cup\{0\}$, the sum in (\ref{psi sum}) is less than the sum $|\beta_{0,n}|\sum_{j=0}^\infty (j+2)!|\lambda|^\frac{j(j-1)}{2} $, which converges to zero as $n\to \infty$ since $|\beta_{0,n}|\overn 0$ and $\sum_{j=0}^\infty (j+2)!|\lambda|^\frac{j(j-1)}{2}$ converges by the ratio test. This proves that $(\psi(T_{\lambda,b}))^n f\overn 0$ uniformly on compact subsets of $\C$, so $\psi(T_{\lambda,b})$ cannot by hypercyclic.
\end{proof}

We summarize the previous two theorems in the following characterization.


\begin{theorem}
Let $\lambda,b\in\C$ and let $T_{\lambda,b}:H(\C)\to H(\C)$ be the operator defined by $T_{\lambda,b}:f(z)\mapsto f'(\lambda z +b)$. The algebra of operators
\begin{align*}
    \{\psi(T_{\lambda,b}): \psi(z)\in H(\C), \psi(0)=0 \text{ and } \psi(T_{\lambda,b}) \text{ is continuous}\}
\end{align*}
consists entirely of hypercyclic operators if $|\lambda|\geq 1$ and consists entirely of non-hypercyclic operators if $|\lambda|>1$.
\end{theorem}

\section{Hypercyclicity of $C_{\lambda,b}\circ \varphi(D)$}\label{last section}

We now look at another generalization of the operators $T_{\lambda,b}=C_{\lambda,b}\circ D$. Let $\varphi(z)$ be an entire function of exponential type, so that the operator $\varphi(D)$ is a convolution operator. We consider in this section the operators $L_{\lambda,b,\varphi}=C_{\lambda,b}\circ \varphi(D)$, each of which is a non-convolution operator whenever $\lambda\not= 1$.   We first prove a type of commutation relation between $C_{\lambda,b}$ and $\varphi(D)$.

\begin{lemma}\label{CommuteLemma}
Suppose $\varphi (D)$ is a convolution operator for some non-constant entire function $\varphi$ of exponential type. Let $\lambda,b\in \C$ with $\lambda \not=0$, and let $C_{\lambda,b}: H(\C)\to H(\C)$ be the composition operator $C_{\lambda,b}:f(z)\mapsto f(\lambda z + b)$. Then $C_{\lambda,b} \circ \varphi(D) = \varphi(\lambda^{-1}D) \circ C_{\lambda,b}  $.

\begin{proof}
Let $\varphi(z)=\sum_{k=0}^\infty w_k z^k$ and let $f(z)\in H(\C)$. Then
\begin{align*}
    \varphi(\lambda^{-1}D)C_{\lambda,b} f(z)=\varphi(\lambda^{-1}D) f(\lambda z+b) & = \sum_{k=0}^\infty w_k \lambda^{-k}D^k (f(\lambda z+b))\\
    &= C_{\lambda,b}\left(\sum_{k=0}^\infty w_k D^k f(z) \right) = C_{\lambda,b}\varphi(D) f(z).
\end{align*}
\end{proof}
\end{lemma}

We now provide yet another family of hypercyclic non-convolution operators. 
\begin{theorem}\label{CompositionTheorem}
Suppose $\varphi (D)$ is a convolution operator for some non-constant entire function $\varphi$ of exponential type with $\varphi (0) = 0$. Let $\lambda,b\in \C$ and let $C_{\lambda,b}: H(\C)\to H(\C)$ be the composition operator $C_{\lambda,b}:f(z)\mapsto f(\lambda z + b)$.  If $|\lambda|\geq 1$, then the operator $L_{\lambda,b,\varphi}=C_{\lambda,b} \circ \varphi(D)$ is hypercyclic.
\end{theorem}

\begin{proof}
We first write $\varphi(z)=z^m \psi(z)$, where $\psi(z)=\sum_{k=0}^\infty w_k z^k$ is an entire function of exponential type with $\psi(0)\not =0$. By repeatedly applying Lemma \ref{CommuteLemma}, we have that
\begin{align}\label{Lnformula} \notag
    L_{\lambda,b,\varphi}^n & = \underbrace{C_{\lambda,b} \varphi(D) C_{\lambda,b} \varphi(D) \cdots C_{\lambda,b} \varphi(D)}_{\text{$n$ times}}\\ \notag
    & =  \varphi\left(\lambda^{-n}D\right) \varphi\left(\lambda^{1-n}D\right) \cdots \varphi\left(\lambda^{-1}D\right) C_{\lambda,b}^n\\ \notag
    & = \frac{D^m}{\lambda^{nm}} \psi\left(\lambda^{-n}D\right) \frac{D^m}{\lambda^{(n-1)m}} \psi\left(\lambda^{1-n}D\right) \cdots \frac{D^m}{\lambda^{m}}\psi\left(\lambda^{-1}D\right) C_{\lambda,b}^n\\ \notag
    & = \psi\left(\lambda^{-n}D\right)  \psi\left(\lambda^{1-n}D\right) \cdots \psi\left(\lambda^{-1}D\right) \frac{1}{\lambda^\frac{nm(n+1)}{2}} D^{nm} C_{\lambda,b}^n \\  \notag
    & = \psi\left(\lambda^{-n}D\right)  \psi\left(\lambda^{1-n}D\right) \cdots \psi\left(\lambda^{-1}D\right) \lambda^\frac{nm(n-1)}{2} C_{\lambda,b}^n  D^{nm}\\
    & = \psi\left(\lambda^{-n}D\right)  \psi\left(\lambda^{1-n}D\right) \cdots \psi\left(\lambda^{-1}D\right) L_{\lambda,b,z^m}^n \text{ by (\ref{L^n eqn})},
\end{align}
where $L_{\lambda,b,z^m}:H(\C)\to H(\C)$ is the operator $L_{\lambda,b,z^m}:f(z)\mapsto f^{(m)}(\lambda z+b)$ considered in Lemma \ref{HallackLemma}.

Let $\mathcal{P}$ be the set of complex polynomials in $H(\C)$, which is a dense subset of $H(\C)$.  Let $p$ be a nonzero polynomial of degree $d$. We define a right-inverse $F_n:\mathcal{P}\to\mathcal{P}$ for $L_{\lambda,b,\varphi}^n$ on $\mathcal{P}$ as follows.  By Lemma \ref{InverseLemma}, there exist $C=C(d)>0$ and constants $a_i\in\C, 1\leq i \leq d$, such that for each positive integer $j$, the mapping $S_{\psi(\lambda^{-j} D)}:\mathcal{P}\to\mathcal{P}$ defined by
\begin{align}\label{Spsi}
    S_{\psi(\lambda^{-j} D)} p =w_0^{-1}\left(I+a_1 \lambda^{-j} D + \cdots + a_d (\lambda^{-j})^d D^d\right)p
\end{align}
is a right-inverse for $\psi(\lambda^{-j} D)$ on $\mathcal{P}$, and $|a_i|<C$ for each $i$.

Let $S_{m,n}:\mathcal{P}\to \mathcal{P}$ be the linear right inverse of $L_{\lambda,b,z^m}^n$ as defined in Lemma \ref{HallackLemma}.   We then define the mapping $F_n:\mathcal{P}\to \mathcal{P}$ by
\begin{align*}
     F_n p = S_{m,n} S_{\psi(D)} \cdots S_{\psi(\lambda^{2-n}D)} S_{\psi(\lambda^{1-n}D)} p,
\end{align*}
which satisfies $L_{\lambda,b,\varphi}^n F_n p =p$.

\sloppy The condition $\varphi(0)=0$ implies $\deg L_{\lambda,b,\varphi} p < \deg p$, which implies $L_{\lambda,b,\varphi}^n p= 0$ whenever $n>d$. What remains to show for the Hypercyclicity Criterion is that 
\begin{align}\label{Fnp}
    F_n p\overn 0
\end{align}
uniformly on compact subsets of $\C$. By multiplying out the product $S_{\psi(D)} \cdots S_{\psi(\lambda^{2-n}D)} S_{\psi(\lambda^{1-n}D)}$ using equation (\ref{Spsi}), we have that
\begin{align*}
    & S_{\psi(D)} \cdots S_{\psi(\lambda^{2-n}D)} S_{\psi(\lambda^{1-n}D)} p\\
    & = w_0^{-n} [I+a_1 D + \cdots + a_d D^d]\cdots [I+a_1 \lambda^{1-n} D + \cdots + a_d (\lambda^{1-n})^d D^d] p\\
    & = w_0^{-n} [I + c_{1,n} D + \cdots + c_{d,n} D^d] p,
\end{align*}
where the coefficients $c_{j,n}$ for $1\leq j\leq d$ satisfy
\begin{align}\label{cjn}
    c_{j,n}=\sum_{j_1+\cdots + j_n = j} \frac{a_{j_1}}{(\lambda^{j_1})^0} \frac{ a_{j_2}}{(\lambda^{j_2})^1}\frac{a_{j_3}}{(\lambda^{j_3})^2}\cdots \frac{a_{j_n}}{(\lambda^{j_n})^{n-1}},
\end{align}
where each $j_k$ is a non-negative integer. 

For each positive integer $n$, let $c_{0,n}=1$. Since $S_{m,n}$ is linear, to show (\ref{Fnp}) it suffices to show that 
\begin{align}\label{Smn}
    S_{m,n}(w_0^{-n}c_{j,n}p)\overn 0.
\end{align}
uniformly on compact subsets of $\C$ for each integer $j$ such that $0\leq j\leq d$. 

Let $j$ be an integer satisfying $0\leq j\leq d$. The number of terms in the sum (\ref{cjn}) is equal to the number of multinomial coefficients in a multinomial sum, which is $\ds\binom{j+n-1}{n-1}$. Let $\alpha = \max\{1,|a_1|,|a_2|,\cdots,|a_d|\}$. Since $|\lambda|\geq 1$, by (\ref{cjn}) we have that 
\begin{align}\label{multinomialexp}
   |c_{j,n}|\leq \binom{j+n-1}{n-1}\alpha^j \leq n^j\alpha^j\leq n^d\alpha^d<e^{nd}\alpha^d.
\end{align}
Now let $t=|w_0|^{-1}e^d\alpha^d$. Then $|w_0^{-n}c_{j,n}|\leq t^n$, and thus $S_{m,n}(w_0^{-n}c_{j,n}p)\overn 0$ uniformly on compact subsets of $\C$ by limit (\ref{Slimit}) in Lemma \ref{HallackLemma}, which shows that (\ref{Smn}) holds and completes the proof.

\end{proof}

\section*{Acknowledgements}
Funding: This work was supported by the St. Olaf College Collaborative Undergraduate Research and Inquiry program.

\end{document}